\documentclass[12pt]{article}
\usepackage[amsmath]{my-e-jc}

\usepackage{graphicx}
\usepackage{xcolor}

\usepackage{xskak}

\dateline{Jan 1, 2024}{Jan 2, 2024}{TBD}


\Copyright{The authors. Released under the CC BY-ND license (International 4.0).}

\title{Domination by kings is oddly even}

\author{Cristopher Moore\authornote{1}
\and
Stephan Mertens\authornote{1,2}
}

\authortext{1}{Santa Fe Institute, Santa Fe 87501 NM, USA,
   (\email{moore@santafe.edu}).}
\authortext{2}{Institut f\"ur Physik, Otto-von-Guericke Universit\"at,  
Magdeburg, Germany (\email{mertens@ovgu.de}).}

\begin{document}

\maketitle


\begin{abstract}
The $m \times n$ king graph consists of all locations on an $m \times n$ chessboard, where edges are legal moves of a chess king. 
Let $P_{m \times n}(z)$ denote its domination polynomial, i.e., $\sum_{S \subseteq V} z^{|S|}$ where the sum is over all dominating sets $S$. We prove that $P_{m \times n}(-1) = (-1)^{\lceil m/2\rceil \lceil n/2\rceil}$. In particular, the number of dominating sets of even size and the number of odd size differs by $\pm 1$. 
This property does not hold for king graphs on a cylinder or a torus, or for the grid graph. But it holds for $d$-dimensional kings, where $P_{n_1\times n_2\times\cdots\times n_d}(-1) = (-1)^{\lceil n_1/2\rceil \lceil n_2/2\rceil\cdots \lceil n_d/2\rceil}$.
\end{abstract}

\section{Introduction}

A placement of chess pieces on a chessboard is called \emph{dominating} if each free square of the chessboard is under attack by at least one piece. chess domination problems have been studied at least since 1862, when Jaenisch~\cite{jaenisch:62} posed the problem of finding the minimum number of queens needed to dominate the $8\times 8$ board. For queens, the answer is 5~\cite{watkins:12}. For kings, the answer is 9 (Fig.~\ref{fig:kings-graph}).

\begin{figure}[ht]
 \begin{tabular}{p{0.5\textwidth}p{0.5\textwidth}}  
  \begin{center}
    \chessboard[setfen=8/1k2k2k/8/8/1k2k2k/8/8/1k2k2k - - - 0 0,
    label=false, showmover=false]
  \end{center} &
   \vspace{6.5mm}
   \begin{center}
     \includegraphics[width=0.362\textwidth]{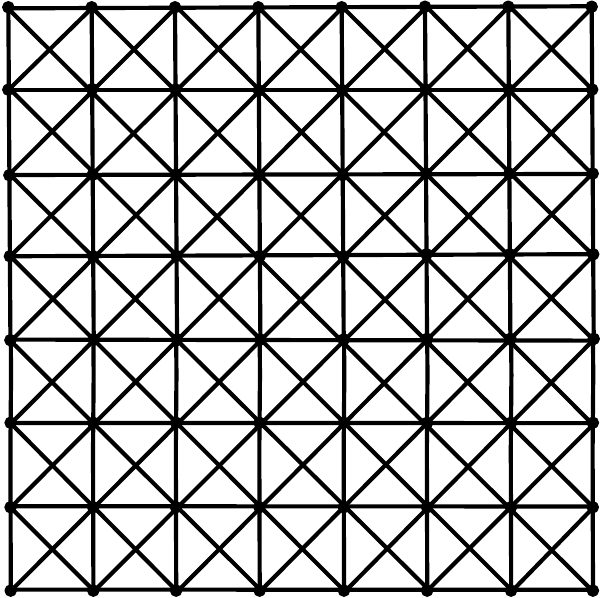}
  \end{center}\\[-1ex]
\end{tabular}
\caption{Nine kings are required to dominate the $8\times 8$
  chessboard. More mathematically: the smallest dominating set of the
  king graph $K_{8\times 8}$ (right) has size $9$.}
\label{fig:kings-graph}
\end{figure}

In the 20th century, the idea of domination entered graph theory~\cite{haynes:hedetniemi:slater:98}.
A dominating set in a graph $G = (V , E)$ is a subset $S \subseteq V$
of vertices such that every node in $V$ is either contained in $S$ or
has a neighbor in $S$.

The link from chess to graph theory is given by graphs like the king
graph (Fig.~\ref{fig:kings-graph}).  In this graph, which we call
$K_{m \times n}$, the vertices represent the squares of an
$m \times n$ chessboard, and each edge represents a legal move of a
king. That is, $K_{m\times n}=(V,E)$ where $V = [m] \times [n]$ and
$((x,y),(x',y')) \in E$ if and only if $\max(|x-x'|,|y-y'|)=1$. Graphs
for other chess pieces can be defined accordingly. For instance,
Rudrata showed in the 9th century C.E. that the $8 \times 8$ knight
graph is Hamiltonian~\cite{dasgupta:algorithms}.

The \emph{domination polynomial} of a graph $G=(V,E)$ is the
generating function of its dominating sets with respect to their size, i.e., 
\begin{equation}
  \label{eq:def_domination_polynomial}
  P_G(z) = \sum_{\substack{S \subseteq V\\ \textrm{$S$ dominating}}} z^{|S|} \;=\; \sum_{k=0}^{|V|} N_k \,z^k\, ,
\end{equation}
where $N_k$ is the number of dominating sets of size $k$. 
Like other graph polynomials, the domination polynomial encodes many
interesting properties of a graph~\cite{akbari:alikhani:peng:10}.  

A closed form of the domination polynomial is known only for a few
very simple graphs like the complete graph, the wheel graph, and the
star graph~\cite{alikhani:peng:14}. Recently, the $m \times n$ rook
graph from chess was added to this list~\cite{mertens:24a}. As far as
we know, the rook is the only chess piece for which this has been
achieved.

In this contribution we consider the domination polynomial
$P_{m \times n}(z)$ of the $m \times n$ king graph.  One of us has
computed $N_k$ and thus $P_{m \times n}(z)$ by exact enumeration for
all $m, n \leq 22$~\cite{mertens:24b}, and the data shows an
interesting pattern when $z=-1$: namely, $P_{m \times n}(-1) = \pm 1$.  Since
\[
P_{m \times n}(-1) = \sum_{\textrm{$k$ even}} N_k - \sum_{\textrm{$k$ odd}} N_k \, , 
\]
this means that the parity of dominating sets is almost perfectly
balanced. Indeed, it is as balanced as it can be, since the total
number of dominating sets in any finite graph is
odd~\cite{brouwer:csorba:schrijver:09}. Specifically, we show that
$P_{m \times n}(-1) = (-1)^{\lceil m/2\rceil \lceil n/2\rceil}$ as in
Table~\ref{tab:kings_skewness}. Thus the number of dominating sets of
odd size is one greater than those of even size if $m$ and $n$ are
both $1$ or $2$ mod $4$, and is one less otherwise.

\begin{table}
  \centering
  {\footnotesize
  \begin{tabular}{r|rrrrrrrrrrrrrrrrrrrrrr}
 & \multicolumn{22}{c}{$m$}\\
$n$ & 1 & 2 & 3 & 4 & 5 & 6 & 7 & 8 & 9 & 10 & 11 & 12 & 13 & 14 & 15 & 16 & 17 & 18 & 19 & 20 & 21 & 22 \\\hline
1 & -1 & -1 & 1 & 1 & -1 & -1 & 1 & 1 & -1 & -1 & 1 & 1 & -1 & -1 & 1 & 1 & -1 & -1 & 1 & 1 & -1 & -1 \\
2 & -1 & -1 & 1 & 1 & -1 & -1 & 1 & 1 & -1 & -1 & 1 & 1 & -1 & -1 & 1 & 1 & -1 & -1 & 1 & 1 & -1 & -1 \\
3 & 1 & 1 & 1 & 1 & 1 & 1 & 1 & 1 & 1 & 1 & 1 & 1 & 1 & 1 & 1 & 1 & 1 & 1 & 1 & 1 & 1 & 1 \\
4 & 1 & 1 & 1 & 1 & 1 & 1 & 1 & 1 & 1 & 1 & 1 & 1 & 1 & 1 & 1 & 1 & 1 & 1 & 1 & 1 & 1 & 1 \\
5 & -1 & -1 & 1 & 1 & -1 & -1 & 1 & 1 & -1 & -1 & 1 & 1 & -1 & -1 & 1 & 1 & -1 & -1 & 1 & 1 & -1 & -1 \\
6 & -1 & -1 & 1 & 1 & -1 & -1 & 1 & 1 & -1 & -1 & 1 & 1 & -1 & -1 & 1 & 1 & -1 & -1 & 1 & 1 & -1 & -1 \\
7 & 1 & 1 & 1 & 1 & 1 & 1 & 1 & 1 & 1 & 1 & 1 & 1 & 1 & 1 & 1 & 1 & 1 & 1 & 1 & 1 & 1 & 1 \\
8 & 1 & 1 & 1 & 1 & 1 & 1 & 1 & 1 & 1 & 1 & 1 & 1 & 1 & 1 & 1 & 1 & 1 & 1 & 1 & 1 & 1 & 1 \\
9 & -1 & -1 & 1 & 1 & -1 & -1 & 1 & 1 & -1 & -1 & 1 & 1 & -1 & -1 & 1 & 1 & -1 & -1 & 1 & 1 & -1 & -1 \\
10 & -1 & -1 & 1 & 1 & -1 & -1 & 1 & 1 & -1 & -1 & 1 & 1 & -1 & -1 & 1 & 1 & -1 & -1 & 1 & 1 & -1 & -1 \\
11 & 1 & 1 & 1 & 1 & 1 & 1 & 1 & 1 & 1 & 1 & 1 & 1 & 1 & 1 & 1 & 1 & 1 & 1 & 1 & 1 & 1 & 1 \\
12 & 1 & 1 & 1 & 1 & 1 & 1 & 1 & 1 & 1 & 1 & 1 & 1 & 1 & 1 & 1 & 1 & 1 & 1 & 1 & 1 & 1 & 1 \\
13 & -1 & -1 & 1 & 1 & -1 & -1 & 1 & 1 & -1 & -1 & 1 & 1 & -1 & -1 & 1 & 1 & -1 & -1 & 1 & 1 & -1 & -1 \\
14 & -1 & -1 & 1 & 1 & -1 & -1 & 1 & 1 & -1 & -1 & 1 & 1 & -1 & -1 & 1 & 1 & -1 & -1 & 1 & 1 & -1 & -1 \\
15 & 1 & 1 & 1 & 1 & 1 & 1 & 1 & 1 & 1 & 1 & 1 & 1 & 1 & 1 & 1 & 1 & 1 & 1 & 1 & 1 & 1 & 1 \\
16 & 1 & 1 & 1 & 1 & 1 & 1 & 1 & 1 & 1 & 1 & 1 & 1 & 1 & 1 & 1 & 1 & 1 & 1 & 1 & 1 & 1 & 1 \\
17 & -1 & -1 & 1 & 1 & -1 & -1 & 1 & 1 & -1 & -1 & 1 & 1 & -1 & -1 & 1 & 1 & -1 & -1 & 1 & 1 & -1 & -1 \\
18 & -1 & -1 & 1 & 1 & -1 & -1 & 1 & 1 & -1 & -1 & 1 & 1 & -1 & -1 & 1 & 1 & -1 & -1 & 1 & 1 & -1 & -1 \\
19 & 1 & 1 & 1 & 1 & 1 & 1 & 1 & 1 & 1 & 1 & 1 & 1 & 1 & 1 & 1 & 1 & 1 & 1 & 1 & 1 & 1 & 1 \\
20 & 1 & 1 & 1 & 1 & 1 & 1 & 1 & 1 & 1 & 1 & 1 & 1 & 1 & 1 & 1 & 1 & 1 & 1 & 1 & 1 & 1 & 1 \\
21 & -1 & -1 & 1 & 1 & -1 & -1 & 1 & 1 & -1 & -1 & 1 & 1 & -1 & -1 & 1 & 1 & -1 & -1 & 1 & 1 & -1 & -1 \\
22 & -1 & -1 & 1 & 1 & -1 & -1 & 1 & 1 & -1 & -1 & 1 & 1 & -1 & -1 & 1 & 1 & -1 & -1 & 1 & 1 & -1 & -1
  \end{tabular}
  }
\caption{Values of $P_{m \times n}(-1)$ from exact enumeration for the
  king graph $K_{m\times n}$, showing the pattern $(-1)^{\lceil m/2\rceil \lceil n/2\rceil}$.}
\label{tab:kings_skewness}
\end{table}

\section{Matching dominating sets of opposite parity}

\begin{figure}[ht]
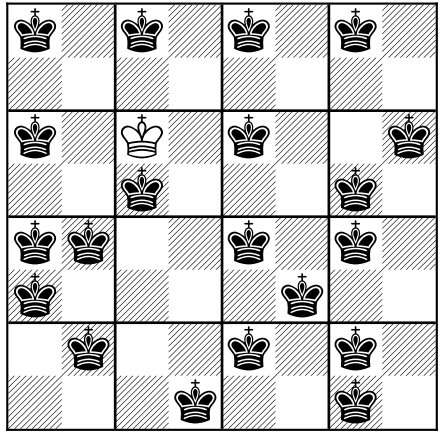

  \begin{center}
      \chessboard[setfen=k1k1k1k1/8/k1K1k2k/2k3k1/kk2k1k1/k4k2/1k2k1k1/3k2k1 - - - 0 0,
    label=false, showmover=false,
pgfstyle=bottomborder,
linewidth=0.1ex,
markregion=a7-h8,
markregion=a5-h6,
markregion=a3-h4,
pgfstyle=rightborder,
linewidth=0.1ex,
markregion=a1-b8,
markregion=c1-d8,
markregion=e1-f8,
    ]
  \end{center} 
\caption{The matching between almost all dominating sets. We scan from left to right and top to bottom until we find the first $2 \times 2$ square where the upper-right, lower-left, or lower-right corner is occupied. We then flip that square's upper-left corner (the white king), adding or removing it from the set. This produces another dominating set of opposite parity.}
\label{fig:kings-matching}
\end{figure}

\begin{figure}[ht]
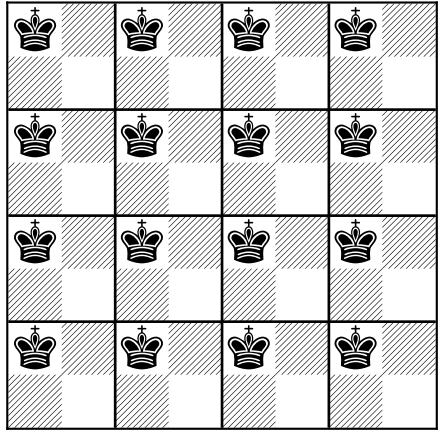

  \begin{center}
          \chessboard[setfen=k1k1k1k1/8/k1k1k1k1/8/k1k1k1k1/8/k1k1k1k1/8 - - - 0 0,
    label=false, showmover=false,
    pgfstyle=bottomborder,
linewidth=0.1ex,
markregion=a7-h8,
markregion=a5-h6,
markregion=a3-h4,
pgfstyle=rightborder,
linewidth=0.1ex,
markregion=a1-b8,
markregion=c1-d8,
markregion=e1-f8,
    ]
  \end{center} 
\caption{The one dominating set that has no partner in our proof of Theorem~\ref{the:skewness}. It has size 16, hence $P_{8 \times 8}(-1)=(-1)^{16}=1$.}
\label{fig:kings-leftover}
\end{figure}

\begin{theorem}
  \label{the:skewness}
  Let $P_{m \times n}(z)$ denote the dominating polynomial of the $m \times n$ king graph $K_{m \times n}$. Then
  \begin{equation}
    \label{eq:skewness}
    P_{m \times n}(-1) = (-1)^{\lceil m/2 \rceil \lceil n/2 \rceil}\,.
  \end{equation}
\end{theorem}
\noindent

\begin{proof}
We use a matching strategy illustrated in Fig.~\ref{fig:kings-matching}. We divide the graph into $2 \times 2$ squares starting at its top-left corner. If $m$ or $n$ is odd, the squares on the right or bottom edge of the graph are truncated to $1 \times 2$, $2 \times 1$, or $1 \times 1$ squares. If $x \in [m]$ and $y \in [n]$ increase to the right and downward so that the top-left vertex of $K_{m \times n}$ is $(1,1)$, the top-left corners of the squares are the vertices $v=(x,y)$ where $x$ and $y$ are both odd.

Now, given a dominating set $S$, we scan the graph from left to right and top to bottom until we find the first $2 \times 2$ square where some vertex other than its top-left corner---that is, some vertex $(x,y)$ where $x$ or $y$ is even---is occupied by a king, i.e., is in $S$. Note that all the squares before that one have their top-left corners occupied, since otherwise their top-left square would be uncovered---i.e., neither occupied nor a neighbor of an occupied vertex---and $S$ would not be dominating. 

Now consider that square's top-left corner $v$, shown as a white king in Fig.~\ref{fig:kings-matching}. It is adjacent only to vertices in the previous squares, which are already covered by their top-left corners, and to the other three vertices in its own square, which are covered by one of the other vertices in that square. Thus we can ``flip'' $v$, adding or removing it from $S$, and the resulting set $S'$ will also be dominating. Applying the same rule gives $S''=S$, so this defines a partial matching of dominating sets. Since $|S|$ and $|S'|$ have opposite parity, the dominating sets with partners in this matching cancel in $P_{m \times n}(-1)$.

The only dominating set for which this procedure does not define a partner in this matching is the one shown in Fig.~\ref{fig:kings-leftover}, consisting of the top-left corner of each square. Call this set $T$. Since all other dominating sets cancel, we have $P_{m \times n}(-1) = (-1)^{|T|}$. Noting that $|T| = (-1)^{\lceil m/2 \rceil \lceil n/2 \rceil}$ completes the proof.
\end{proof}

\section{King graphs on cylinders and tori}

Our proof of Theorem~\ref{the:skewness} relies on free boundary conditions. If we consider the king graph on a cylinder or torus by making $x$ and/or $y$ cyclic, our matching strategy no longer works---there is no way to define the ``first'' $2 \times 2$ square. In fact, the statement of the theorem no longer holds. When we compute $P_{m \times n}(-1)$ by exact enumeration on cylinders and tori, we find it is no longer restricted to $\pm 1$ (see
Table~\ref{tab:kingtorus}).

\begin{table}
  \centering
  \begin{tabular}{r|rrrr}
 & \multicolumn{4}{c}{$m$}\\
$n$ & 3 & 4 & 5 & 6 \\\hline
3 & 5 & -3 & -1 & -1 \\
4 & 3 & 39 & 11 & 43 \\
5 & 9 & 33 & -1 & -1 \\
6 & 15 & -13 & -1 & 11
  \end{tabular}
  \hspace{20mm}
\begin{tabular}{r|rrrr}
 & \multicolumn{4}{c}{$m$}\\
$n$ & 3 & 4 & 5 & 6 \\\hline
3 & -1 & 3 & -1 & -1 \\
4 & 3 & 63 & 3 & 51 \\
5 & -1 & 3 & -1 & -1 \\
6 & -1 & 51 & -1 & 11
\end{tabular}
  \caption{Values of $P_{m \times n}(-1)$ for the king graph on the cylinder where $x \in [m]$ is cyclic (left) and on the torus where both $x\in[m]$ and $y\in[n]$ are cyclic (right).}
  \label{tab:kingtorus}
\end{table}

\section{Higher dimensions}

The king graph can be generalized to other dimensions, including chess in three and more dimensions~\cite{pritchard:07}.
Given $n_1, \ldots, n_d \ge 0$, we define $K_{n_1\times n_2 \times \cdots \times n_d}$ as $G=(V,E)$ where $V=[n_1] \times \cdots [n_d]$ and $(u,v) \in E$ if $\max_{i \in [d]} |u_i-v_i| = 1$. Equivalently, 
$K_{n_1\times n_2 \times \cdots \times n_d}$ is the strong graph product of $d$ path graphs of size $n_1,\ldots,n_d$. Then we have the following:

\begin{theorem}
 \label{the:skewness}
  Let $d \ge 1$ and left $P_{n_1\times n_2 \times \cdots \times n_d}(z)$ denote the dominating polynomial of the $d$-dimensional king graph $K_{n_1\times n_2 \times \cdots \times n_d}$. Then
  \begin{equation}
    \label{eq:skewness}
    P_{n_1\times n_2 \times \cdots \times n_d}(-1) = (-1)^{\lceil n_1/2 \rceil \lceil n_2/2 \rceil \cdots \lceil n_d/2 \rceil} \, .
  \end{equation}
\end{theorem}

\begin{proof}
The same matching strategy works, but now with $d$-dimensional $2 \times \cdots \times 2$ hypercubical blocks, truncated at the right, bottom, rear, et cetera boundary if the corresponding $n_i$ is odd. Each block has an upper-left-frontmost-et ceterest corner $v$ where $v_1, \ldots, v_d$ are all odd. Scan left to right, top to bottom, front to back, ana to kata~\cite{rucker:03} and so on, finding the first block where some vertex other than this $v$ is occupied. Then flipping $v$ gives another dominating set $S'$ of opposite parity. The only unmatched dominating set is the set $T$ consisting of all these corners $v$, and $|T|=\prod_{i=1}^d \lceil n_i / 2 \rceil$.
\end{proof}

\section{Conclusions}

\begin{table}
  \begin{center}
    {\footnotesize
   \begin{tabular}{r|rrrrrrrrrrrrrrrr}
 & \multicolumn{16}{c}{$m$}\\
$n$ & 1 & 2 & 3 & 4 & 5 & 6 & 7 & 8 & 9 & 10 & 11 & 12 & 13 & 14 & 15 & 16 \\\hline
1 & -1 & -1 & 1 & 1 & -1 & -1 & 1 & 1 & -1 & -1 & 1 & 1 & -1 & -1 & 1 & 1 \\
2 & -1 & 3 & -3 & 5 & -5 & 7 & -7 & 9 & -9 & 11 & -11 & 13 & -13 & 15 & -15 & 17 \\
3 & 1 & -3 & 1 & 5 & -3 & -3 & 5 & 1 & -3 & 1 & 1 & 1 & 1 & -3 & 1 & 5 \\
4 & 1 & 5 & 5 & 5 & 1 & 1 & -3 & 1 & 1 & 9 & 5 & 9 & 1 & 1 & -7 & 1 \\
5 & -1 & -5 & -3 & 1 & 3 & -1 & 9 & 13 & -1 & -5 & -3 & -3 & -13 & -9 & 5 & 5 \\
6 & -1 & 7 & -3 & 1 & -1 & 15 & -7 & 9 & 3 & 3 & -23 & 25 & -5 & -5 & -11 & 29 \\
7 & 1 & -7 & 5 & -3 & 9 & -7 & -3 & -3 & 1 & -11 & 17 & 1 & 1 & -15 & 21 & 13 \\
8 & 1 & 9 & 1 & 1 & 13 & 9 & -3 & 9 & 5 & 5 & 1 & 21 & -7 & 1 & 29 & 9 \\
9 & -1 & -9 & -3 & 1 & -1 & 3 & 1 & 5 & 23 & -13 & -3 & -7 & -9 & -9 & 29 & 21 \\
10 & -1 & 11 & 1 & 9 & -5 & 3 & -11 & 5 & -13 & 27 & -19 & 41 & 19 & 31 & -15 & 25 \\
11 & 1 & -11 & 1 & 5 & -3 & -23 & 17 & 1 & -3 & -19 & 1 & -15 & 17 & -11 & 9 & -35 \\
12 & 1 & 13 & 1 & 9 & -3 & 25 & 1 & 21 & -7 & 41 & -15 & 29 & -23 & 41 & -19 & 45 \\
13 & -1 & -13 & 1 & 1 & -13 & -5 & 1 & -7 & -9 & 19 & 17 & -23 & 19 & -37 & 21 & -35 \\
14 & -1 & 15 & -3 & 1 & -9 & -5 & -15 & 1 & -9 & 31 & -11 & 41 & -37 & 103 & -39 & 77 \\
15 & 1 & -15 & 1 & -7 & 5 & -11 & 21 & 29 & 29 & -15 & 9 & -19 & 21 & -39 & 53 & -107 \\
16 & 1 & 17 & 5 & 1 & 5 & 29 & 13 & 9 & 21 & 25 & -35 & 45 & -35 & 77 & -107 & 169
\end{tabular}}
  \end{center}
  \caption{$P_{m\times n}(-1)$ from exact enumeration for the
    $m\times n$ grid graph (or the wazir graph if you like Shatranj).}
  \label{tab:wazir}
\end{table}

From a physics point of view, the domination polynomial $P(z)$ is the partition function of a spin system at zero temperature where the energy is the number of non-dominated vertices. For the king graph, this system can also be thought of as sets of $3 \times 3$ squares (the kings' neighborhoods) which can overlap but which must cover every vertex. The parameter $z$ is then the fugacity of this system. We find it curious that this partition function can be computed exactly for $z=-1$. It would be wonderful if we could compute $P(z)$ at other values of real or complex $z$. For instance, a three-way grouping between sets of size $0$, $1$, and $2$ mod $3$ could determine $P(z)$ where $z$ is a cube root of $1$.

For what other systems does this kind of argument exist? As part of a study of percolation~\cite{mertens:moore:19}, we performed exact enumeration of spanning configurations on grid graphs of various sizes, i.e., sets of vertices that contain a path from one side of the grid to the other. These numerical results suggested that the number of such sets of odd and even size differ by $\pm 1$. To prove this, we defined a partial matching similar to the one in this paper by flipping a vertex to produce another spanning configuration of opposite parity, leaving a single configuration whose size determines $P(-1)$. As for kings, that matching generalizes to higher-dimensional cubic lattices, and to some other graphs as well. 

Is this phenomenon sporadic or widespread? We have already seen that, for the king graph, the matching strategy does not work on cylinders or tori. As a counterexample with free boundary conditions, in Table~\ref{tab:wazir} we show $P(-1)$ for grid graphs of various sizes, and see that $P(z)$ is not restricted to $\pm 1$. 

In chess terms these are wazir graphs, where the wazir can move only one step north, south, east, or west~\cite{fairy-chess}. The problem seems to be that the wazir does not have a partial neighborhood that tiles the grid in the same way the $2 \times 2$ square does. But perhaps a matching strategy could work on some region in the grid other than a rectangle---such as a region that can be tiled with right trominoes.

\section*{Acknowledgments}

We are grateful to the solar eclipse of April 2024 for bringing us back together. The title of this paper is inspired by a Sudoku designer known as oddlyeven.



\bibliographystyle{unsrt}
\bibliography{math,mertens,cs}

\end{document}